\documentclass{article}

\usepackage[a4paper, dvips]{geometry}
\usepackage{latexsym}
\usepackage{amsfonts}
\usepackage{amsmath, amssymb,amsthm}
\usepackage{bbm}
\usepackage{color}
\usepackage{tikz}
\usetikzlibrary{calc}
\usepackage{cite}
\usepackage{enumerate}
\usepackage{authblk}
\usepackage{graphicx}

\newtheorem{thm}{Theorem}[section]
\newtheorem{corollary}[thm]{Corollary}
\newtheorem{proposition}[thm]{Proposition}

\newtheorem{lemma}[thm]{Lemma}

\newcommand{\ba}{\mathbf{a}}
\newcommand{\bb}{\mathbf{b}}
\newcommand{\bm}{\mathbf{m}}
\newcommand{\bo}{\mathbf{o}}

\newcommand{\bp}{\mathbf{p}}

\newcommand{\bx}{\mathbf{x}}

\newcommand{\NN}{\mathbb{N}}

\newcommand{\RR}{\mathbb{R}}

\def\centerarc[#1](#2)(#3:#4:#5)%
{ \draw[#1] ($(#2)+({#5*cos(#3)},{#5*sin(#3)})$) arc (#3:#4:#5); }

\numberwithin{equation}{section}

\title{Explicit lower bounds for opaque sets\\ of unit square and unit disc}

\author[1]{Markus Kiderlen\footnote{The work of MK was supported by the
		 Aarhus University Research Foundation, under the travel grant AUFF-E-2024-6-13}}
\author[2]{Florian Pausinger\footnote{The work of FP is funded by FCT - Fundação para a Ciência e a Tecnologia, I.P., through national funds, under the project UID/04561/2025}}
\affil[1]{Aarhus University, Aarhus, Denmark}
\affil[2]{Faculdade de Ci\^{e}ncias da Universidade de Lisboa, Lisbon, Portugal}

\date{}
\begin{document}

\maketitle
\begin{abstract}
Explicit lower bounds for the length of the shortest opaque set for the unit disc and the unit square in the Euclidean plane are derived. The results are based on an explicit application of the general method of Kawamura, Moriyama, Otachi and Pach \cite{pach}. 
Employing a recent observation by Steinerberger on the possible orientations of straight barriers with length close to Jones' bound, we improve the bound in  \cite{pach} by more than a factor $3$. The
bound for barriers of the unit disc is new and based on the idea that the free parameters in the general method from \cite{pach} can be optimized due to the strong symmetry properties of the disc. Our approach illustrates both the power and the limitations of the method. \\[1ex]
\end{abstract}

\section{Introduction}
Let  a compact convex set $K \subset \RR^2$ be given. 
A set $B \subseteq \RR^2$ is called a \emph{barrier} or an \emph{opaque set} for $K$ if any line that intersects $K$ also intersects $B$. Note that parts of a barrier can lie outside $K$ and barriers need not be connected in general.

Since there always are barriers with finite length  for a given bounded set $K$ (consider e.g.~the boundary of $K$), 
one can ask for the \emph{shortest} barrier for a given $K$. 
In the case of polygons, this question was already asked more than a century ago by Mazurkiewicz \cite{maz}. Despite  the simplicity of its statement the problem is intriguingly difficult and  no answer is known even for very simple sets $K$ such as a square, a disc or an equilateral triangle. 

One way to approach this problem is to find lower bounds for the shortest length of a barrier, and we will give prominent examples below. To obtain upper bounds for that length, `good'  barriers for a given $K$ have been constructed. Kawohl \cite{kawohl} gives a comprehensive overview of the general problem of finding barriers and presents different constructions for shapes like the unit square and the unit disc. Furthermore, there is a linear time algorithm (linear in $n$) to find a connected polygonal barrier whose length is at most 1.5716 times longer than the optimal barrier of a given convex polygon $P$ with $n$ vertices; see \cite{pach1}. 

In the following, we state this problem in more concise terms. 
We will restrict considerations to \emph{rectifiable} barriers in the sense of \cite{pach}, i.e.~barriers that can be written as unions of at most countably many rectifiable curves. We therefore are interested in 
\[
b(K)=\inf\{ |B|: B \text{ is a rectifiable barrier for } K\}, 
\] 
where $|\cdot|$ denotes length. 
It is shown in \cite[Lemma 4]{pach} that for any rectifiable barrier $B$ there is a \emph{straight} barrier (a finite union of line segments) with a length that is arbitrarily close to $|B|$. Thus, we also have 
\[
b(K)=\inf\{ |B|: B \text{ is a straight barrier for } K\}, 
\] 
so we can and will work with straight barriers. 
It is open if this infimum is attained, even if ``straight'' is replaced by ``rectifiable''. However, if the infimum is taken over (not necessarily straight) closed rectifiable barriers having at most $k$ connected components for some fixed $k\in \NN$, a minimal barrier exists, see \cite{faber} or the outline of the argument in \cite{kawohl}. We will not need the existence of a shortest (straight) barrier
for what follows.

The determination of strong lower bounds for $b(K)$ is a difficult problem. For a long time 
the best lower bound was the one determined by Jones  in 1962 (see \cite{jones}) for the unit square and states that the length of any rectifiable barrier $B$ of a compact convex set $K$ with perimeter $2p$ is at least $p$; 
see \cite[Lemma 1]{pach} for a short proof based on the Crofton formula.
The following theorem and its constructive proof improves on this bound and is the starting point for the subsequent calculations. 
\begin{thm}[Theorem 3, \cite{pach}]\label{thmPach}
For any compact convex set $K$ with perimeter $2p$ that is not a triangle, there is $\delta=\delta(K) > 0$ such that every rectifiable barrier of $K$ has length at least $p+\delta$.
\end{thm}
For arbitrary compact convex sets $K$ the proof of this theorem can be used as a general framework for working out explicit lower bounds improving Jones' result. This was done in \cite{pach} for the centered unit square $K=[-1/2,1/2]^2$.
The triangle case was partially resolved in \cite{izumi} using the method from \cite{pach} together with an additional idea, the so-called restricted barriers.

The following table gives a list of known and new lower bounds for the lengths of rectifiable barriers for two simple sets $K$.
\begin{center}
	\begin{tabular}{  l | c | c  r | c | c r }
		\hline
		& peri- &  best known & [ref.] & new lower & best known &
		[ref.] \\[0ex] 
		shape & meter &  lower bound &  &  bound (Thm.~2)& upper bound&\\[0.5ex] 
		\hline &&&&&&\\[-2.2ex] &&&&&$\sqrt2+\sqrt6/2$&\\
		unit square & $4$ & $2+2\cdot 10^{-5}$ & \cite{pach}
		&
		$2+ 6.3\cdot 10^{-5}$&
		$\approx 2.639$& \cite{jones62}\\  
		\hline&&&&&&\\[-2.2ex]
		unit disc & $2\pi$ & $\pi$ & \cite{croft2} &$\pi + 1.076 \cdot 10^{-6}$ &$\approx 4.799$&\cite{makai}\\  
		\hline
	\end{tabular}
\end{center}

The case of the unit disk is also known as the beam detection constant problem \cite[Sect.~8.11]{finch}; it is listed as problem A30 in \cite{croft} and discussed in a 1995 issue of Scientific American \cite{stewart}, see also \cite{joris, wieacker}.

Our main results are the following lower bounds. 

\begin{thm} \label{thm:main}
Let $B$ be a rectifiable 	barrier of $K\subset\RR^2$. 

\begin{itemize}
	\item[(i)] 	If $K$ is the unit square, we have $|B|>   2 + 6.3  \cdot 10^{-5}$.
	\item[(ii)] If $K$ is the unit disc, we have $|B|>  \pi + 1.076 \cdot 10^{-6}$.
	
\end{itemize}
\end{thm}

The proof of these results follows very tightly  \cite{pach}, where the currently best explicit lower bound for the
unit square, is derived. We refine their arguments slightly to get a better bound for the unit square. This is mainly done to outline where we follow and where we deviate from their approach. In the case of the unit square, we enhance the method in  \cite{pach} using the observation in \cite{steinerberger} that a straight barrier with a length close to the Jones bound must consist of line segments that are approximately parallel to the sides of the square. This is the main source of the improvement in Theorem \ref{thm:main} compared to \cite{pach}. The case of the unit disc makes the proof of \cite[Theorem 3]{pach} explicit in this special case. When $K$ is the unit disc, the set $K_\zeta$ (defined later) is simply a disc and it is thus possible to optimize the parameters in the proof. 

The purpose of this paper is not only to give explicit lower bounds for barriers of particular convex bodies but also to show the potential and limitations of the approach in \cite{pach}. This approach only uses the barrier property in a small neighborhood of a few selected boundary points, leaving the 'central part' of the convex body unexploited. Therefore, even a finer analysis based on this method will not supply lower bounds that are by magnitudes better than the ones derived here. In particular, to close or at least narrow the gap to the best known upper bounds, one would have to improve the crucial 
\cite[Lemma 7]{pach} (quoted below, see Lemma \ref{lem7}), among others  to exploit the local orientation of the line segments in a straight barrier better. 

The paper is organized as follows. 
In Section \ref{sec2} we summarize the main results and the method of \cite{pach}. Section \ref{sec3} presents first our construction for the unit square (a mere variation of \cite{pach}, enhanced with ideas from \cite{steinerberger}). 
We then turn to the unit disc, introducing our construction and deriving an optimization problem that exploits the method in the best possible way. 

\section{The method of Kawamura, Moriyama, Otachi and Pach}
\label{sec2}

Before we outline the method, we need some preliminary notation. 
Unless otherwise stated angles of lines are always understood with respect to the horizontal axis and taken modulo $2\pi$. 

\begin{figure}[h]
	\begin{center}
		\begin{tikzpicture}[scale=0.75]  
			\draw [black,fill=black!20] plot [smooth cycle] coordinates {(11.1,0.3) (12.8,1.1) (12.4,3.1) (11.4,4.1) (10.5, 4.5) (9.5,4.5) (8.7, 4.2) (8.2,2.5) (8.5,1)};
			\node at (10,3) {$A^+$}; 
			\draw [black,fill=black!20] plot [smooth cycle] coordinates {(3.1,-0.1) (4.1,1.3) (4.4,3.3)  (1.5,4.7) (0.7, 4.4)  (0.5,1.2)};
			\node at (2,3) {$A^-$}; 
			
			\draw[black] (4.5,5) --(6.5,2.5) -- (8.5,0);
			\draw[black] (5.5,2.5) --(6.5,2.5) -- (7.5,2.5);
			\draw[black] (4.5,0) --(6.5,2.5) -- (8.5,5);
			
			\node at (6.5,2.5) {$\bullet$}; 
			
			\node at (6,2.8) {$\lambda$};
			\node at (7,2.8) {$\lambda$};
			
			\node at (8.7,4.8) {$g^+$};
			\node at (4.25,4.8) {$g^-$};
			
		\end{tikzpicture}
	\end{center}
	
	\caption{\label{fig:sep} Two $\lambda$-separated sets $A^+$ and $A^-$.}
\end{figure}

\begin{enumerate}
\item For an acute angle $\lambda$,  two sets $A^+, A^- \subset \RR^2$ are \emph{$\lambda$-separated} if there are lines $g^{\pm}$ with angles $\pm \lambda$, respectively, such that $A^-$ is strictly below $g^-$ and strictly above $g^+$ and $A^+$ is strictly  above $g^-$ and strictly below $g^+$; see Fig.~\ref{fig:sep} for an illustration.
\item For $X \subset \RR^2$ and an angle $\alpha\in \RR$ we define the set 
$$X(\alpha)=\{x \sin \alpha - y \cos \alpha: (x,y) \in X \} \subseteq \RR. $$
This is the \emph{projection} of $X$ onto the line whose normal has angle $\alpha$, if this line is isometrically identified with the real line $\RR$. This identification is used here, as we will only be interested in metric properties (lengths) of projections.  
\item Given $K \subseteq \RR^2$ and an angle $\zeta \in [0, \pi)$, 
we let $K_\zeta$ be the set of all points $\bp\in \RR^2$ such that $K$, seen from $\bp$, appears under an angle at least $\pi-\zeta$. 
More concisely, $\bp \in \RR^2\setminus K_\zeta$ iff there is a convex cone with apex $\bp$ and angle of size $\pi - \zeta$ that contains $K$ in its interior.  
\end{enumerate}

The method developed in \cite{pach} leading to Theorem \ref{thmPach} is based on three lemmas.
The first observation \cite[Lemma 5]{pach} is that in order to improve the bound of Jones, it suffices to find a part $B'\subseteq B$ of the barrier whose contribution to covering $K$ is too small. 
\begin{lemma} \label{lem5}
Let $B$ be a rectifiable barrier of a compact convex set $K$ of perimeter $2p$. If there is a subset $B'\subseteq B$ with
$$\int_{-\pi}^{\pi} \left |B'(\alpha) \cap K(\alpha) \right| \, d\alpha \leq 4 |B'| - 4 \delta$$
then $|B| \geq p+\delta$. 
\end{lemma}

The authors then describe two different ways in which such waste can occur. The first possibility \cite[Lemma 6]{pach} is that a significant part of the barrier lies \emph{far} outside of $K$:
\begin{lemma} \label{lem6}
Let $K \subseteq \RR^2$ be a compact convex body and $\zeta \in [0, \pi)$. For any rectifiable set $B'\subset \RR^2\setminus K_{\zeta}$, we have
$$\int_{-\pi}^{\pi} \left |B'(\alpha) \cap K(\alpha) \right| \, d\alpha \leq 4 |B'| \cos \tfrac{\zeta}{2}.$$
\end{lemma}

As it will be used frequently later on,  we notice the following combination of the last two lemmas.  
\begin{corollary}\label{coro1}
	Let $K \subseteq \RR^2$ be a compact convex body with perimeter $p$, $\zeta \in [0, \pi)$ and let $B$ be a rectifiable barrier for $K$. For any rectifiable set ${B}'\subset B$ outside of $K_{\zeta}$, we have
\[
| B | -p\geq |B'|\big(1-\cos \tfrac{\zeta}{2}\big). 
\]
\end{corollary}
\begin{proof}
	Let a rectifiable set $B'\subset B$ outside $K_{\zeta}$ be given. According to Lemma \ref{lem6}, the integral in the displayed formula of this lemma is bounded from above by 
	\[
	 4 |B'| \cos \tfrac{\zeta}{2}=
	 4 |B'|-4|B'|\big(1-\cos \tfrac{\zeta}{2}\big).  
	\]
	Lemma \ref{lem5} with $\delta=|B'|\big(1-\cos \frac{\zeta}{2}\big)$ now gives the assertion. 
\end{proof}
The second possibility for a significant waste 
is based on the observation that a hypothetical barrier for $K$ with length equal to half the perimeter would intersect (almost) every line at most once. Hence, pairs of subsets of a general barrier that 
face each other in the following precise sense will lead to multiple 
intersections with lines. The authors in \cite[Lemma 7]{pach} were able to quantify the resulting  waste for straight barriers.
\begin{lemma} \label{lem7}
For an acute angle $\lambda \in (0,\pi/2)$, a positive number $\eta > 0$, and a $\lambda$-separated pair of compact sets $R^-, R^+ \subseteq \RR^2$, there is $\delta=\delta(\lambda, \eta, R^-, R^+)>0$ such that the following holds:
Assume that $B^-\subset R^-$ and $B^+\subset R^+$ satisfy $|B^-|, |B^+| > \eta$ and that they are unions of line segments that make  angles at least $\lambda$ with the horizontal axis. Then
$$\int_{-\pi}^{\pi} \left | (B^-\cup B^+)(\alpha) \right| \, d\alpha \leq 4 |{B}^- \cup {B}^+| - 4 \delta.$$
\end{lemma}
A quantitative version of the last lemma \cite[Lemma 7']{pach} exploits the fact that $\lambda$-separated sets can be separated by 
strips of positive width, the latter appearing  explicitly in the  constant $\delta$. We state this result, again combined with Lemma \ref{lem5}. 
\begin{corollary}\label{coro2}
		Let  $\eta>0$, $\lambda \in (0,\pi/2)$, $\gamma\in (0,\lambda)$ and $R^-,R^+\subset \RR^2$ be two compact sets contained in a disc of radius $D>0$ and such that they can be separated by strips with angle $\pm(\lambda - \gamma)$ and width $\eta \sin \gamma$. 
		
		If a rectifiable barrier $B$ of $K$ contains two rectifiable sets $B^\pm\subset R^\pm$ of length at least $\eta > 0$ each,
		that are unions of line segments that make  angles at least $\lambda$ with the horizontal axis, 
		 then 
		\[
		| B | -p\geq  \frac{(\eta \sin \gamma)^2}{2D}. 
		\]
\end{corollary}

The idea of the proof of the lower bound can now be summarized as follows. Consider four different exposed  points $\bx_0, \ldots, \bx_3$ on the boundary of $K$, i.e.~for each $i=0,\ldots,3$ there is a tangent $g_i$  of $K$ touching $K$ exactly in $\bx_i$. 
Consider the union of $g\cap K_\zeta$, where $g$ runs through all lines that hit $K$, are parallel to $g_i$, and have a distance at most $s>0$ from $g_i$. 
Let  $R_i$ contain this set, $i=0,\ldots,3$, see Fig.~\ref{fig:Q} for the case $K=[-1/2,1/2]^2$.
If $\lambda\in(0,\pi)$ and $s>0$ are chosen small enough, any two $R_i$'s are $\lambda$-separated (possibly after a suitable rotation of $K$).

Now any line parallel to $g_i$ hitting $R_i$ must hit $K$ and hence the barrier. 
Thus, the part of the barrier that is hit by those lines has total length at least $s$.   
By construction, the points of this part are in $K_\zeta^C$ (and thus negligible by Corollary \ref{coro1}) or in $R_i$. 
So, for each $i$, we have $|B\cap R_i|\ge 2\eta$ for some $\eta>0$. 
By a combinatorial argument it can be shown that among the parts of the barrier that lie in different $R_i$'s, there are segments of total length $\eta$ that face each other and thus lead to waste by Corollary \ref{coro2}.

We illustrate this general principle first with the centered unit square, $K=[-1/2,1/2]^2$, and apply it to the unit disc $K=U$, for which calculations are actually simpler due to the strong symmetry properties.

\section{Two applications}
\label{sec3}

To be able to apply the method to the two different sets $K$, we write $sA=\{sa:a\in A\}$ for the scaling of $A\subset \RR^2$ with $s\ge 0$, and $$[\ba,\bb]=\{s\ba+(1-s)\bb:0\le s\le 1\}$$ for the line segment in $\RR^2$ with endpoints $\ba, \bb\in \RR^2$. The usual inner product in $\RR^2$  is $\langle\cdot,\cdot\rangle$. We will also  write $\bx^\perp$ for the line orthogonal to a vector $\bx\in \RR^2,\bx\ne \bo$.

\subsection{Revisiting the  unit square}
This subsection illustrates the general approach with the known example of the centered unit square $K=Q=[-1/2,1/2]^2$, simplifying the choice of the sets $R_i$ and improving upon the lower bound in \cite{pach}.  

	The main goal of this subsection  is to show Theorem \ref{thm:main}.(i), and illustrate how the ideas from \cite{pach} work for barriers of the unit cube.  
Following the idea outlined in the previous section, the first step is to determine the set $Q_{\zeta}$ associated with $Q$, and to choose four appropriate points $\bx_0, \ldots, \bx_3$ on the boundary of $Q$.
In a second step, we determine the four sets $R_i$ in accordance with the description at the end of Sect.~\ref{sec2}, and show that they are $\lambda$-separated for a suitable $\lambda>0$. This sets the stage for the application of the two corollaries from Section \ref{sec2}, i.e., for showing that we either have a sufficiently large part $B_{\mathrm{out}}=B\setminus Q_{\zeta}$ of the barrier outside $Q_{\zeta}$, or, alternatively, that a sufficiently large part of the barrier inside the sets $R_i$ contains line segments that intersect too many lines more than once. For didactic reasons, we first describe how minor modifications of the concrete procedure  for $Q$ in Sect.~\ref{sec2} lead to a slight improvement stated in Proposition \ref{prop:pure} below. After that, we refine the analysis employing \cite{steinerberger} to obtain a better lower bound in Theorem \ref{thm:main}.(i). 
We start with a geometric description of the set $Q_\zeta$, see 
	Fig.~\ref{fig:Q}, left. 	
\begin{lemma}
	Let $K=Q=[-1/2,1/2]^2$ and $\zeta \in [0,\pi/2)$ be given. Then the boundary of
	$Q_\zeta$ is the union of four circular arcs of radius  $1/(2\sin\zeta)$ that meet in the vertices of $Q$. 
	The arc connecting $\bx_0$ and $\bx_1$ is part of a circle with midpoint 
	$
	\big(0,(1-\cot\zeta)/2\big)^\top,
	$ 
	and the midpoints of the other arcs are determined by symmetry. 
	
	If $\zeta=\pi/4$, the set $Q_\zeta$ is the circumdisc of $Q$ (the smallest disc containing $Q$). 
\end{lemma}
\begin{proof}
Fix an angle $\zeta \in [0,\pi/2)$. No triangle with one side of $Q$ as base can contain $Q$ completely. Triangles erected over diagonals of $Q$ with an angle at least $\pi-\zeta> \pi/2$ cannot contain any of the other two vertices of  $Q$. 
Hence, if $\bp$ is a point in the boundary of $Q_\zeta$,  there is a side $\bx_i\bx_{i+1}$ of $Q$ such that the triangle $\bx_i\,\bp\,\bx_{i+1}$ has angle  $\varphi=\pi-\zeta$ at $\bp$ (indices are understood modulo $4$ here). The inscribed angle theorem implies that the locus of all points with this property lies on a circle with radius $r=r(\zeta)=1/(2\sin\varphi)=1/(2\sin\zeta)$. This circle contains $\bx_i$ and $\bx_{i+1}$, and its midpoint and $Q$  are on the same side of $\bx_i\bx_{i+1}$. For instance, for $i=0$, the midpoint of this circle is 
\[
\bx_1+r\cdot (\sin\zeta,-\cos\zeta)^\top=
\tfrac12\big(0,1-\cot\zeta\big)^\top,
\] 
as asserted. 
\end{proof}

\begin{figure}[h]
	\begin{center}
		\begin{tikzpicture}[scale=3]  
			\draw[fill=black!15]
		(0.5,0.5) node[anchor=west]{$\bx_0$}
		-- (-0.5,0.5) node[anchor=east]{$\bx_1$}
		-- (-0.5,-0.5) node[anchor=east]{$\bx_2$}
		-- (0.5,-0.5) node[anchor=west]{$\bx_3$}
		-- cycle;
		\draw[red] (0.6,0) node[anchor=west] {$Q_\zeta$};
		\draw (-0.05,0) node[anchor=west] {$Q$};

		\draw[red,domain=340:380,samples=200] plot ({1.46*sin(\x)}, {-.873+1.46*cos(\x)});

			\begin{scope}[rotate around={90:(0,0)}]
				\draw[red,domain=340:380,samples=200] plot ({1.46*sin(\x)}, {-.87+1.46*cos(\x)});		
			\end{scope}
			
			\begin{scope}[rotate around={180:(0,0)}]
				\draw[red,domain=340:380,samples=200] plot ({1.46*sin(\x)}, {-.87+1.46*cos(\x)});		
			\end{scope}
			
			\begin{scope}[rotate around={270:(0,0)}]
				\draw[red,domain=340:380,samples=200] plot ({1.46*sin(\x)}, {-.87+1.46*cos(\x)});		
			\end{scope}
			
		\end{tikzpicture}
		\qquad
		%
		\begin{tikzpicture}[scale=3]  

			\draw[fill=black!15]
			(0.5,0.5) node[anchor=west]{\color{cyan!70!black} $R_0$}
			-- (-0.5,0.5) node[anchor=east]
			 {\color{cyan!70!black} $R_1$}
			-- (-0.5,-0.5) node[anchor=east]
			{\color{cyan!70!black} $R_2$}
			-- (0.5,-0.5) node[anchor=west]
			{\color{cyan!70!black} $R_3$}
			-- cycle;
			\draw[red] (0.6,0) node[anchor=west] {$Q_\zeta$};
			\draw (-0.05,0) node[anchor=west] {$Q$};

			\begin{scope}[rotate around={90:(0,0)}]
				\filldraw[cyan!90!black,domain=340:350,samples=200] 
				plot ({1.46*sin(\x)}, {-.873+1.46*cos(\x)})
				-- (135:0.58)--cycle;
				\filldraw[cyan!90!black,domain=370:380,samples=200] 
				plot ({1.46*sin(\x)}, {-.873+1.46*cos(\x)})
				-- (45:0.58)--cycle;
				\draw[red,domain=340:380,samples=200] plot 			({1.46*sin(\x)}, {-.873+1.46*cos(\x)});
				\draw [blue,<-] (65:0.64) .. controls (66:0.7) and 	(73:.78) .. (70:.8) node[anchor=east]{$\bp_{1-}$};
				
				\draw [blue,<-] (24:0.63) .. controls (26:0.7) and 	(22:.75) .. (27:.8) node[anchor=east]{$\bp_{1+}$};
				
			\end{scope}

			\filldraw[cyan!90!black,domain=340:350,samples=200] 
			plot ({1.46*sin(\x)}, {-.873+1.46*cos(\x)})
			-- (135:0.58)--cycle;
		    	\filldraw[cyan!90!black,domain=370:380,samples=200] 
		    plot ({1.46*sin(\x)}, {-.873+1.46*cos(\x)})
		    -- (45:0.58)--cycle;
			\draw[red,domain=340:380,samples=200] plot ({1.46*sin(\x)}, {-.873+1.46*cos(\x)});
		    \draw [blue,<-] (65:0.64) .. controls (60:0.9) and (58:.8) .. (55:.8) node[anchor=west]{$\bp_{0+}$};
		    
		    \draw [blue,<-] (26:0.65) .. controls (27:0.7) and 
		    (27:.75) .. (26:.8) node[anchor=west]{$\bp_{0-}$};
		    \draw  ({1.46*sin(350)}, {-.873+1.46*cos(350)}) -- 
		    ( {-(-.873+1.46*cos(350))},{-1.46*sin(350)});	
		    \draw (138:0.55) node[anchor=west] {$\ell_1$};

			\begin{scope}[rotate around={180:(0,0)}]
				\filldraw[cyan!90!black,domain=340:350,samples=200] 
			plot ({1.46*sin(\x)}, {-.873+1.46*cos(\x)})
			-- (135:0.58)--cycle;
			\filldraw[cyan!90!black,domain=370:380,samples=200] 
			plot ({1.46*sin(\x)}, {-.873+1.46*cos(\x)})
			-- (45:0.58)--cycle;
			\draw[red,domain=340:380,samples=200] plot 		({1.46*sin(\x)}, {-.873+1.46*cos(\x)});

			\end{scope}
			
			\begin{scope}[rotate around={270:(0,0)}]
								\filldraw[cyan!90!black,domain=340:350,samples=200] 
				plot ({1.46*sin(\x)}, {-.873+1.46*cos(\x)})
				-- (135:0.58)--cycle;
				\filldraw[cyan!90!black,domain=370:380,samples=200] 
				plot ({1.46*sin(\x)}, {-.873+1.46*cos(\x)})
				-- (45:0.58)--cycle;
				\draw[red,domain=340:380,samples=200] plot 		({1.46*sin(\x)}, {-.873+1.46*cos(\x)});
				\draw  ({1.46*sin(350)}, {-.873+1.46*cos(350)}) -- 
				( {-(-.873+1.46*cos(350))},{-1.46*sin(350)});	
				
				\draw (130:0.58) node[anchor=east] {$\ell_0$};
			\end{scope}

			\draw[very thin, dashed]
			(0.5,0.5) -- (-0.5,0.5) -- (-0.5,-0.5)
			-- (0.5,-0.5) -- cycle;
			
		\end{tikzpicture}
	\end{center}
	
	\caption{\label{fig:Q} Left: The centered unit square, $Q$ and the associated set $Q_\zeta$ for $\zeta=\pi/9$. Right: the sets $R_i$ with line segments $\ell_i$ in their boundaries. }
\end{figure}

The following proposition results from an application of the method developed in \cite{pach} with minor improvements. Clearly, the minimal deviation from Jones' bound in Theorem \ref{thm:main}.(i) is stronger, and Proposition \ref{prop:pure} and its proof are only stated to recall the steps in \cite{pach} and the adjustments suggested by us. 
\begin{proposition} \label{prop:pure}
	If  $B$ is a rectifiable 	barrier of the unit square $Q\subset\RR^2$, we must have 
	\[
	|B|>   2 + 2.3 \cdot 10^{-5}.
	\]
\end{proposition}

To outline the proof of Proposition \ref{prop:pure}, we
denote the vertices of $Q$ as in \cite{pach} by
\[
\bx_0=\tfrac1{2}(1,1)^\top,\quad
\bx_1=\tfrac1{2}(-1,1)^\top,\quad
\bx_2=\tfrac1{2}(-1,-1)^\top,\quad
\bx_3=\tfrac1{2}(1,-1)^\top 
\]
and choose $1/\sqrt2> t>1-1/(2\sqrt2)$. Let  $g_i=\bx_i+\bx_i^\perp$ be the line trough $\bx_i$, orthogonal to the line connecting the origin $\bo$ and $\bx_i$, and define 
\[
R_i=\{\bx\in Q_\zeta:t\le {\sqrt 2}\langle \bx,\bx_i\rangle \le \tfrac1{\sqrt2}\}
\] 
$i=0,\ldots,3$. The sets $R_i$ are slightly smaller than those defined on \cite[p.~18]{pach} and (given $g_i$  and the width $1/\sqrt2-t$) the smallest possible for the method to work. We define the line segments $\ell_i= \{\bx\in Q_\zeta:t= {\sqrt 2}\langle \bx,\bx_i\rangle\}$ and denote its endpoints by $\bp_{i+}$ and $\bp_{i-}$, see Fig.~\ref{fig:Q}.
The coordinates of these points can be obtained intersecting the line 
containing $\ell_i$ with the appropriate circular arc in the boundary of $Q_\zeta$. For $i=1$ we get $\bp_{1-}=(x,y)$ with 
\[
x=\tfrac{1}{2}\Big(c-\sqrt2 t+\sqrt{2r^2-(c-\sqrt2 t)^2}\,\Big),
\quad y=\sqrt2t+x,
\]
with $c=(1-\cot\zeta)/2$ and $r=1/(2\sin\zeta)$. 
The other points $\bp_{i\pm}$ can then be found using symmetry arguments. 

One can now redo essentially the proof in \cite[Theorem 2]{pach} with $\zeta=0.15$
and $\lambda=\pi/8$  considering two cases. If $|B_{\mathrm{out}}| > 0.008182$, 
Corollary 1 with  $B'=B_{\mathrm {out}}$ implies directly the claim $|B|-2>2.3\cdot 10^{-5}$, implying Proposition \ref{prop:pure} in this case. 
	
If $|B_{\mathrm{out}}|=|B\setminus Q_{\zeta}| \le  0.008182$,  one may proceed as in the proof of  \cite[Theorem 2]{pach}.

\begin{figure}[h]
	\begin{center}
		\begin{tikzpicture}[scale=3]  
			\draw[fill=black!15]
			(0.5,0.5) node[anchor=west]{\color{cyan!70!black} $R_0$}
			-- (-0.5,0.5) node[anchor=east]
			{\color{cyan!70!black} $R_1$}
			-- (-0.5,-0.5) node[anchor=east]
			{\color{cyan!70!black} $R_2$}
			-- (0.5,-0.5) node[anchor=west]
			{\color{cyan!70!black} $R_3$}
			-- cycle;
			\draw (-0.05,0) node[anchor=west] {$Q$};
			
			\begin{scope}[rotate around={90:(0,0)}]
				\filldraw[cyan!90!black,domain=340:350,samples=200] 
				plot ({1.46*sin(\x)}, {-.873+1.46*cos(\x)})
				-- (135:0.58)--cycle;
				\filldraw[cyan!90!black,domain=370:380,samples=200] 
				plot ({1.46*sin(\x)}, {-.873+1.46*cos(\x)})
				-- (45:0.58)--cycle;
				\draw[gray,domain=340:380,samples=200] plot 			({1.46*sin(\x)}, {-.873+1.46*cos(\x)});
			\end{scope}

			\filldraw[cyan!90!black,domain=340:350,samples=200] 
			plot ({1.46*sin(\x)}, {-.873+1.46*cos(\x)})
			-- (135:0.58)--cycle;
			\filldraw[cyan!90!black,domain=370:380,samples=200] 
			plot ({1.46*sin(\x)}, {-.873+1.46*cos(\x)})
			-- (45:0.58)--cycle;
			\draw[gray,domain=340:380,samples=200] plot ({1.46*sin(\x)}, {-.873+1.46*cos(\x)});

			\begin{scope}[rotate around={180:(0,0)}]
				\filldraw[cyan!90!black,domain=340:350,samples=200] 
				plot ({1.46*sin(\x)}, {-.873+1.46*cos(\x)})
				-- (135:0.58)--cycle;
				\filldraw[cyan!90!black,domain=370:380,samples=200] 
				plot ({1.46*sin(\x)}, {-.873+1.46*cos(\x)})
				-- (45:0.58)--cycle;
				\draw[gray,domain=340:380,samples=200] plot 		({1.46*sin(\x)}, {-.873+1.46*cos(\x)});		
			\end{scope}
			
			\begin{scope}[rotate around={270:(0,0)}]
				\filldraw[cyan!90!black,domain=340:350,samples=200] 
				plot ({1.46*sin(\x)}, {-.873+1.46*cos(\x)})
				-- (135:0.58)--cycle;
				\filldraw[cyan!90!black,domain=370:380,samples=200] 
				plot ({1.46*sin(\x)}, {-.873+1.46*cos(\x)})
				-- (45:0.58)--cycle;
				\draw[gray,domain=340:380,samples=200] plot 		({1.46*sin(\x)}, {-.873+1.46*cos(\x)});
				\draw  ({1.46*sin(350)}, {-.873+1.46*cos(350)}) -- 
				( {-(-.873+1.46*cos(350))},{-1.46*sin(350)});	
			\end{scope}
			
			\draw[very thin, dashed]
			(0.5,0.5) -- (-0.5,0.5) -- (-0.5,-0.5)
			-- (0.5,-0.5) -- cycle;
			
			\draw[thick] (0.45,0.47)--(0.25,0.47);
			\draw[thick] (0.47,0.45)--(0.47,0.25);
			\draw[red,thick] (0.455,0.455)--(0.32,0.32);
			\begin{scope}[rotate around={90:(0,0)}]
				\draw[red,thick] (0.45,0.47)--(0.25,0.47);
				\draw[thick] (0.47,0.45)--(0.47,0.25);
				\draw[thick] (0.455,0.455)--(0.32,0.32);
			\end{scope}
			\begin{scope}[rotate around={180:(0,0)}]
				\draw[red,thick] (0.45,0.47)--(0.25,0.47);
				\draw[thick] (0.47,0.45)--(0.47,0.25);
				\draw[thick] (0.455,0.455)--(0.32,0.32);
			\end{scope}
			\begin{scope}[rotate around={270:(0,0)}]
				\draw[thick] (0.45,0.47)--(0.25,0.47);
				\draw[thick] (0.47,0.45)--(0.47,0.25);
				\draw[red,thick] (0.455,0.455)--(0.32,0.32);
			\end{scope}
			
		\end{tikzpicture}
		\qquad
		\begin{tikzpicture}[scale=3]  
			\draw[fill=black!15]
			(0.5,0.5) node[anchor=west]{\color{cyan!70!black} $R_0$}
			-- (-0.5,0.5) node[anchor=east]
			{\color{cyan!70!black} $R_1$}
			-- (-0.5,-0.5) node[anchor=east]
			{\color{cyan!70!black} $R_2$}
			-- (0.5,-0.5) node[anchor=west]
			{\color{cyan!70!black} $R_3$}
			-- cycle;
			\draw (-0.05,0) node[anchor=west] {$Q$};
			
			\begin{scope}[rotate around={90:(0,0)}]
				\filldraw[cyan!90!black,domain=340:350,samples=200] 
				plot ({1.46*sin(\x)}, {-.873+1.46*cos(\x)})
				-- (135:0.58)--cycle;
				\filldraw[cyan!90!black,domain=370:380,samples=200] 
				plot ({1.46*sin(\x)}, {-.873+1.46*cos(\x)})
				-- (45:0.58)--cycle;
				\draw[gray,domain=340:380,samples=200] plot 			({1.46*sin(\x)}, {-.873+1.46*cos(\x)});
			\end{scope}

			\filldraw[cyan!90!black,domain=340:350,samples=200] 
			plot ({1.46*sin(\x)}, {-.873+1.46*cos(\x)})
			-- (135:0.58)--cycle;
			\filldraw[cyan!90!black,domain=370:380,samples=200] 
			plot ({1.46*sin(\x)}, {-.873+1.46*cos(\x)})
			-- (45:0.58)--cycle;
			\draw[gray,domain=340:380,samples=200] plot ({1.46*sin(\x)}, {-.873+1.46*cos(\x)});

			\begin{scope}[rotate around={180:(0,0)}]
				\filldraw[cyan!90!black,domain=340:350,samples=200] 
				plot ({1.46*sin(\x)}, {-.873+1.46*cos(\x)})
				-- (135:0.58)--cycle;
				\filldraw[cyan!90!black,domain=370:380,samples=200] 
				plot ({1.46*sin(\x)}, {-.873+1.46*cos(\x)})
				-- (45:0.58)--cycle;
				\draw[gray,domain=340:380,samples=200] plot 		({1.46*sin(\x)}, {-.873+1.46*cos(\x)});		
			\end{scope}
			
			\begin{scope}[rotate around={270:(0,0)}]
				\filldraw[cyan!90!black,domain=340:350,samples=200] 
				plot ({1.46*sin(\x)}, {-.873+1.46*cos(\x)})
				-- (135:0.58)--cycle;
				\filldraw[cyan!90!black,domain=370:380,samples=200] 
				plot ({1.46*sin(\x)}, {-.873+1.46*cos(\x)})
				-- (45:0.58)--cycle;
				\draw[gray,domain=340:380,samples=200] plot 		({1.46*sin(\x)}, {-.873+1.46*cos(\x)});
				\draw  ({1.46*sin(350)}, {-.873+1.46*cos(350)}) -- 
				( {-(-.873+1.46*cos(350))},{-1.46*sin(350)});	
			\end{scope}
			
			\draw[very thin, dashed]
			(0.5,0.5) -- (-0.5,0.5) -- (-0.5,-0.5)
			-- (0.5,-0.5) -- cycle;
			
			\draw[red,thick] (0.45,0.47)--(0.25,0.47);
			\draw[thick] (0.47,0.45)--(0.47,0.25);
			\draw[thick] (0.455,0.455)--(0.32,0.32);
			\begin{scope}[rotate around={90:(0,0)}]
				\draw[red,thick] (0.45,0.47)--(0.25,0.47);
				\draw[thick] (0.47,0.45)--(0.47,0.25);
				\draw[thick] (0.455,0.455)--(0.32,0.32);
			\end{scope}
			\begin{scope}[rotate around={180:(0,0)}]
				\draw[red,thick] (0.45,0.47)--(0.25,0.47);
				\draw[thick] (0.47,0.45)--(0.47,0.25);
				\draw[thick] (0.455,0.455)--(0.32,0.32);
			\end{scope}
			\begin{scope}[rotate around={270:(0,0)}]
				\draw[red,thick] (0.45,0.47)--(0.25,0.47);
				\draw[thick] (0.47,0.45)--(0.47,0.25);
				\draw[thick] (0.455,0.455)--(0.32,0.32);
			\end{scope}
			
		\end{tikzpicture}
	\end{center}
	\caption{\label{fig:newQ} Symbolic display of each of the sets $B_{i,j}$ by one line segment with representative orientation.	The red line segments represent sets $B_{i,j}$ with $|B_{i,j}|>\eta$.
	In order to apply Corollary \ref{coro2}, we aim for a pair $(i_0,j_0)$ with $i_0<j_0$ such that $B_{i_0,j_0}$ and $B_{j_0,i_0}$ are represented by black line segments.\\
		Left:  Corollary \ref{coro2} can only be applied to the pairs of barrier subsets in \emph{neighboring} $R_i$'s: choose either subsets of $(B_0,B_1)$ or $(B_0,B_3)$. 
		Right: Corollary \ref{coro2} can only be applied to the pairs of barrier subsets in \emph{opposing} $R_i$'s: choose either subsets of $(B_0,B_2)$ or $(B_1,B_3)$. 
	}
\end{figure}

For the reader's convenience, we outline the main arguments from  that paper in the following. 
	Define $B_i=B\cap R_i$, $i=0,\ldots,3$.
	The union of all lines  $g$ parallel to $\ell_i$ hitting $Q\cap R_i$ is a strip of width $1/\sqrt2-t$. Any line in this strip must hit $Q$ and thus the barrier in $B_i\cup B_{\mathrm{out}}$, implying $|B_i\cup B_{\mathrm{out}}|\ge 1/\sqrt2-t$. Since both sets are disjoint, we have $|B_i|\ge 2\eta$, $i=0,1,2,3$, with 
	\begin{equation}\label{eq:etaoutQ}
		\eta=\tfrac{1}2\big(\tfrac1{\sqrt2}-t-|B_{\mathrm{out}}|\big).
	\end{equation}
	In \cite{pach}, each of the sets $B_i$ is paritioned  into three subsets $B_{i,j}$,
	$i=0,1,2,3$, $j\in \{0,1,2,3\}\setminus \{i\}$. The set $B_{i,j}$ is the union of all line segments in $B_i$ that point approximately  in direction of the set $R_j$ (i.e.~with an angle deviating at most $\lambda=\pi/8$ from that direction). For instance, $B_{0,1}$ contains the union of all line segments in $B_0$ with a slope 
	deviating at most by $\pi/8$ from the horizontal direction. 
	(We remark the notational difference that the authors of \cite{pach} work with \emph{sets} of line segments while we work with their union, so  our sets $B_{i,j}$ are subsets of $\mathbb R^2$; for details, also the formal definition of the sets $B_{i,j}$, we refer to \cite[p.~18]{pach}.) Fig.~\ref{fig:newQ} illustrates this by representing each of the sets  $B_{i,j}$ by \emph{one} line segment at $R_i$ pointing towards the corresponding set $R_j$. The purpose of this figure is to support intuition, it does \emph{not} actually draw the sets $B_{i,j}$. The latter sets might consist of many line segments and are contained in one of the $R_i$'s. 
	If one wants to apply Corollary \ref{coro2} to (rotations of) $R_i$ and $R_j$, one must ensure that $|B_{i,j}|$ and $|B_{j,i}|$ are not too large, as otherwise $B_i$ and $B_j$ consist predominantly of roughly parallel line segments. This is achieved in \cite{pach} by a combinatorial argument: Since $|B_i|\ge 2\eta$, there is at most one  $j\ne i$ with $|B_{i,j}|>\eta$.  Hence, among the six possible pairs $(i,j)$, $i<j$, there are at most four such that 
	either $|B_{i,j}|>\eta$ or $|B_{j,i}|>\eta$. The remaining pairs $(i_0,j_0)$ satisfy $|B_{i_0}\setminus B_{i_0,j_0}|> \eta$ and  $|B_{j_0}\setminus B_{j_0,i_0}|> \eta$, and  Corollary \ref{coro2} can be applied to them. We let $B^+$ and $B^-$ be appropriate rotations of the sets $B_{i_0}\setminus B_{i_0,j_0}$ and $B_{j_0}\setminus B_{j_0,i_0}$, respectively. An application of Corollary \ref{coro2} to the sets $R^+$ and $R^-$ (obtained 
	by the same rotation of $R_{i_0}$ and $R_{j_0}$ about the origin) and subsets $B^+$ and $B^-$ of barriers with total lengths at least $\eta$ in them.  

However, we apply three changes to improve the above procedure. Firstly, slightly smaller sets $R_i$ are used. 
Secondly, we use the parameters $(\zeta,t)=(0.15,0.615)$ (instead of the tuple $(0.1,0.619)$ from \cite[p.~18]{pach}), where the width of their $R_i$ was $\sqrt2/16$ corresponding to our $1/\sqrt2-t$) and get $\eta = 
(1/\sqrt 2-0.615 - 0.008182)/2 \approx 0.0419$.
Thirdly, when separating  pairs of neighboring sets $R_i$, for instance $R_0$ and $R_1$, say,  we use the fact that they are contained in a disc centered at $(\bx_0+\bx_1)/2$ of radius $0.522$
and thus use $D=1.044$ instead of the value $\sqrt{2}\frac{41}{40} \approx 1.45$ in \cite[p.~18]{pach}.
Corollary \ref{coro2} with angle $\gamma=0.19$
gives the claim $|B|-2\ge 3\cdot 10^{-5}$ in this case.
However, when separating \emph{opposing} sets $R_i$, 
	 the above  value of  $D$ is too small, so we use 
	 $D=1.415<\sqrt2$ since $Q_\zeta$ is contained in the circumdisc of $Q$.  Applied to the rotations $R^-,R^+$ of two opposing sets $R_i$, the assumptions of Corollary \ref{coro2} turn out to hold with $\gamma=0.194$ and yield again $|B|-2> 2.3\cdot 10^{-5}$. The assertion of Proposition \ref{prop:pure} is thus shown. \medskip 
	 
	 We now explain how the above arguments can be adjusted based on a recent development to obtain the better bound of Theorem \ref{thm:main}.(i). With the same parameters $(\zeta,t,\lambda)=(0.15,0.615,\pi/8)$ as before, consider two cases. If $|B_{\mathrm{out}}| >0.022410$, 
	 Corollary 1 with  $B'=B_{\mathrm {out}}$ implies directly the claim $|B|-2>6.3\cdot 10^{-5}$, implying the claim in Theorem \ref{thm:main}.(i) in this case.
	 Assume from now on that $|B_{\mathrm{out}}|\le  0.022410$ holds.
	 
	 Steinerberger introduced an `angular orientation measure' $\mu_L$ of finite unions $L\subset \mathbb R^2$ of line segments. For a segment $s$ of length $a$ enclosing an angle $\alpha\in[0,\pi)$ with the horizontal axis, this measure is defined as  $\mu_s=\frac12(\delta_\alpha+\delta_{\alpha+\pi})$. Here, $\delta_\alpha$ is the usual Dirac probability measure
	 supported by $\{\alpha\}$.  If $L$ is the union of the line segments $s_1,\ldots,s_k$ whose pairwise  intersections either are empty or singletons, he defines $\mu_L=\sum_{i=1}^k\mu_{s_i}$. The main result of  	
	 \cite{steinerberger} states that for  a straight barrier $B$ of a convex polygon $P$ (even a compact convex set in $\mathbb R^2$) with a length close to Jones' bound, $\mu_B$ and $\frac12\mu_{\partial P}$ are close in the norm of an appropriate homogeneous Sobolev space. This is quantified in \cite[Theorem]{steinerberger} and applied to the unit square $P=Q$, where it implies that the line segments constituting such a barrier must typically be (almost) vertical or horizontal. More specific,	\cite[Proposition]{steinerberger} states that if $B$ is a straight barrier for $Q$ with $|B|\le 2+\delta$, and $0\le \Lambda\le \pi/4$ is given, then 
	 \begin{equation}
	 	\label{eq:steinerberger}
	 	\mu_{B}(A_\Lambda)\le \frac{\delta}{1-\cos\Lambda},
	 \end{equation}
	 where $$A_\Lambda=[0,2\pi)\setminus\bigcup_{j=0}^3 \Big[\frac j2\pi-\Lambda, \frac j2\pi+\Lambda\Big]$$ is the set af all angles deviating at most by $\Lambda$ from the horizontal ($\alpha\in \{0,\pi\}$) or the vertical ($\alpha\in\{\pi/2,3\pi/2\}$) directions.  Hence, the total length of all line segments in $B$ with directions in $A_\Lambda$ is at most $\delta/(1-\cos\Lambda)$ in this case.

	 \begin{figure}[h]
	 	\begin{center}
	 		\begin{tikzpicture}[scale=3]  
	 			\draw[fill=black!15]
	 			(0.5,0.5) node[anchor=west]{\color{cyan!70!black} $R_0$}
	 			-- (-0.5,0.5) node[anchor=east]
	 			{\color{cyan!70!black} $R_1$}
	 			-- (-0.5,-0.5) node[anchor=east]
	 			{\color{cyan!70!black} $R_2$}
	 			-- (0.5,-0.5) node[anchor=west]
	 			{\color{cyan!70!black} $R_3$}
	 			-- cycle;
	 			\draw (-0.05,0) node[anchor=west] {$Q$};
	 			
	 			\begin{scope}[rotate around={90:(0,0)}]
	 				\filldraw[cyan!90!black,domain=340:350,samples=200] 
	 				plot ({1.46*sin(\x)}, {-.873+1.46*cos(\x)})
	 				-- (135:0.58)--cycle;
	 				\filldraw[cyan!90!black,domain=370:380,samples=200] 
	 				plot ({1.46*sin(\x)}, {-.873+1.46*cos(\x)})
	 				-- (45:0.58)--cycle;
	 				\draw[gray,domain=340:380,samples=200] plot 			({1.46*sin(\x)}, {-.873+1.46*cos(\x)});
	 			\end{scope}

	 			\filldraw[cyan!90!black,domain=340:350,samples=200] 
	 			plot ({1.46*sin(\x)}, {-.873+1.46*cos(\x)})
	 			-- (135:0.58)--cycle;
	 			\filldraw[cyan!90!black,domain=370:380,samples=200] 
	 			plot ({1.46*sin(\x)}, {-.873+1.46*cos(\x)})
	 			-- (45:0.58)--cycle;
	 			\draw[gray,domain=340:380,samples=200] plot ({1.46*sin(\x)}, {-.873+1.46*cos(\x)});

	 			\begin{scope}[rotate around={180:(0,0)}]
	 				\filldraw[cyan!90!black,domain=340:350,samples=200] 
	 				plot ({1.46*sin(\x)}, {-.873+1.46*cos(\x)})
	 				-- (135:0.58)--cycle;
	 				\filldraw[cyan!90!black,domain=370:380,samples=200] 
	 				plot ({1.46*sin(\x)}, {-.873+1.46*cos(\x)})
	 				-- (45:0.58)--cycle;
	 				\draw[gray,domain=340:380,samples=200] plot 		({1.46*sin(\x)}, {-.873+1.46*cos(\x)});		
	 			\end{scope}
	 			
	 			\begin{scope}[rotate around={270:(0,0)}]
	 				\filldraw[cyan!90!black,domain=340:350,samples=200] 
	 				plot ({1.46*sin(\x)}, {-.873+1.46*cos(\x)})
	 				-- (135:0.58)--cycle;
	 				\filldraw[cyan!90!black,domain=370:380,samples=200] 
	 				plot ({1.46*sin(\x)}, {-.873+1.46*cos(\x)})
	 				-- (45:0.58)--cycle;
	 				\draw[gray,domain=340:380,samples=200] plot 		({1.46*sin(\x)}, {-.873+1.46*cos(\x)});
	 				\draw  ({1.46*sin(350)}, {-.873+1.46*cos(350)}) -- 
	 				( {-(-.873+1.46*cos(350))},{-1.46*sin(350)});	
	 			\end{scope}
	 			
	 			\draw[very thin, dashed]
	 			(0.5,0.5) -- (-0.5,0.5) -- (-0.5,-0.5)
	 			-- (0.5,-0.5) -- cycle;
	 			
	 			\draw[thick] (0.45,0.47)--(0.25,0.47);
	 			\draw[thick] (0.47,0.45)--(0.47,0.25);
	 			\begin{scope}[rotate around={90:(0,0)}]
	 				\draw[thick] (0.45,0.47)--(0.25,0.47);
	 				\draw[thick] (0.47,0.45)--(0.47,0.25);
	 				\draw[green!80!black,thick] (0.455,0.455)--(0.32,0.32);
	 			\end{scope}
	 			\begin{scope}[rotate around={180:(0,0)}]
	 				\draw[thick] (0.45,0.47)--(0.25,0.47);
	 				\draw[thick] (0.47,0.45)--(0.47,0.25);
	 			\end{scope}
	 			\begin{scope}[rotate around={270:(0,0)}]
	 				\draw[thick] (0.45,0.47)--(0.25,0.47);
	 				\draw[thick] (0.47,0.45)--(0.47,0.25);
	 				\draw[green!80!black,thick] (0.455,0.455)--(0.32,0.32);
	 			\end{scope}
	 			
	 		\end{tikzpicture}	
	 	\end{center}
	 	\caption{\label{fig:newSteinerb} Symbolic display of each of the sets $B_{i,j}$ by one line segment with representative orientation, where the barriers  $B_{0,2}$ and $B_{2,0}$ with directions almost coinciding with the main diagonal have been neglected due to an application of  \cite[Proposition]{steinerberger}, see main text. Actually, also the barrier subsets represented by the green directions have a small length, namely at most $2\theta$ each, but this is not needed in the proof.
	 	}
	 \end{figure}
	 
	 We apply this to the barrier of the unit square in order to control the lengths of some of the sets  $B_{i,j}$ from above. 
	 We get bounds in the case where $i$ and $j$ have the same parity, corresponding to opposing sets $B_i$ and $B_j$, i.e.~where line segments are roughly parallel to one of the diagonals of $Q$. If $|B|> 2+\delta$ with $\delta=6.3  \cdot 10^{-5}$, the statement in Theorem \ref{thm:main}.(i) holds, so we may assume $|B|\le 2 + \delta$. 
	 The bound \eqref{eq:steinerberger} with $\Lambda=\lambda=\pi/8$ 
	 implies 
	 \[
	 \sum_{i\ne j, i+j\text{ even}} |B_{i,j}|\le \frac{\delta}{1-\cos\lambda}\le 8.28\cdot 10^{-4}=:2\theta. 
	 \]
	 Thus, either $\theta\ge |B_{0,2}|+|B_{2,0}|$ or $\theta \ge |B_{1,3}|+|B_{1,3}|$. Without loss of generality, assume that 
	 \[
	 \theta \ge |B_{0,2}|+|B_{2,0}|\ge \max\{|B_{0,2}|,|B_{2,0}|\},
	 \]
	  implying 
	 \[
	   |B_{0}\setminus B_{0,2}|> \eta' \text{ and }  
	   |B_{2}\setminus B_{2,0}|> \eta'.
	 \]
	 Here, $\eta'= 2\eta-\theta=0.06969-4.14\cdot 10^{-4}\ge0.06928$, where \eqref{eq:etaoutQ}, $|B_{\mathrm{out}}|\le  0.022410$ and $\theta=4.14\cdot 10^{-4}$ were used. The geometric configuration is depicted in Fig.~\ref{fig:newSteinerb}. As the sets $R_0$ and $R_2$ haven't changed, Corollary \ref{coro2} can be applied as before, but with the new $\eta'$ instead of $\eta$, yielding 
	 \[
	 |B|-2\ge \frac{(\eta'\sin(\gamma))^2}{2D}\ge 6.3\cdot 10^{-5}, 
	 \]
	 where $\gamma=0.194$ and $D=1.415$, just like in the proof of Proposition \ref{prop:pure} for opposing sets $R_i$. This concludes the proof of Theorem \ref{thm:main}.(i).

 \subsection{The unit disc} 
Let $U \subset \RR^2$ be the unit disc centered at the origin $\bo$. 
Recall that $U_\zeta$ is the set of all points that are not  $\zeta$-far from $U$, where  $\zeta\in [0,\pi)$. In our case, any  point $\bp$ on the boundary of $U_{\zeta}$ is an apex point of a suitable  convex cone $C_\bp$ of angle  $\pi - \zeta$ at the apex point, such that $U\subset C_\bp$ hits both boundary rays of $C_\bp$.

Following again the general proof idea, the first step is to explore the set $U_{\zeta}$ associated with $U$: due to the strong symmetry properties of $U$, the set $U_{\zeta}$ is a disc and we establish a relation between its radius and the angle $\zeta$.
Next, we choose four appropriate points $\bx_0, \ldots, \bx_3$ on the boundary of $U$
and  determine the four sets $R_i$ in accordance with the description at the end of Sect.~\ref{sec2}. The relatively simple relations between the free parameters now allow us to optimize them  in order to maximize the gain obtained from the application of the two corollaries in the final step.\medskip 

We start with a description of $U_\zeta$.

\begin{proposition}\label{prop:Uzeta}
	Let $U \subset \RR^2$ be the unit disc and $0\le \zeta<\pi$. Then $U_\zeta$ is a disc containing $U$, and if $ U_\zeta=(1+r)U$, $r\ge 0$, we have 
	\begin{equation}
		\label{eq:zeta_r}
		\zeta=2\arccos\Big(\frac1{1+r}\Big).
	\end{equation}
\end{proposition}
\begin{proof}
Clearly, the rotational symmetry of $U$ implies that $U_\zeta$ is a disc centered at the origin with radius $1+r$ for some  $r\ge 0$. 
To make the relation between $r$ and $\zeta=\zeta(r)$ explicit, consider the two tangents of $U$ passing through the point  $\bp=(0,1+r)^\top$ in the boundary of $U_\zeta$; see Fig.~\ref{fig:circle}, left.  These two lines touch $U$ in the two points $\bx_\pm$ of the intersection of the boundary of $U$ with the Thales circle with center $\bp/2$ and radius $(1+r)/2$. The isosceles triangle $\bo,\bp/2, \bx_+$ thus has two sides of length $(1+r)/2$ and one of length $1$, so its angle at $\bo$ is $\alpha=\arccos(1/(1+r))$. Thus, the right triangle $\bo,\bp,\bx_+$ has angle $\beta=\pi/2-\alpha$ at $\bp$, and so 
\begin{equation*}
	\zeta(r)=\pi-2\beta=2\arccos\Big(\frac1{1+r}\Big).
\end{equation*}
This concludes the proof. 
\end{proof}
The next result shows in particular that parts of the barrier that are very far away from $U$ (i.e.~in cases where $r$ is very large) contribute almost with their full lengths to the deviation from Jones' bound.

 \begin{proposition}\label{prop2}
 	Let $r\ge 0$, let  $B$ be a barrier of the unit disc $U$ and define  $B_{\mathrm{out}}=B\setminus [(1+r)U]$. Then 
\begin{equation}\label{eq:Bout}
	|B|-\pi\ge \frac{r}{1+r}|B_{\mathrm{out}}|.
\end{equation} 	
 \end{proposition}
 \begin{proof}
 	According to Proposition \ref{prop:Uzeta}, we have $(1+r)U=U_\zeta$ with $\zeta$ given by \eqref{eq:zeta_r}, so $\cos\tfrac\zeta2=1/(1+r)$.  Corollary \ref{coro1} with $B'=B_{\mathrm{out}}$ now gives 
 	\[
 	|B|-\pi\ge (1-\cos\tfrac\zeta2)|B_{\mathrm{out}}|=\frac{r}{1+r}|B_{\mathrm{out}}|,
 	\]
 	as claimed. 
 \end{proof}
 
\begin{figure}[h]
	\begin{center}
	
	\begin{tikzpicture}[scale=1.5]
	\draw[fill=black!5] (0,0) circle (1);	
			\filldraw (0,0) circle (.4pt);
			\draw (0,0)  circle (1);
			\draw (-0.5,-0.3) node[anchor=west] {$U$};
			\draw[color=red] (0,0) 	circle (1.2);
			\draw[color=red] (0,-1.2) node[anchor=north] {$U_\zeta=(1+r)U$};
			
			\draw [blue,<-] (57.8:1.05) .. controls (65:1.6) and (55:1.3) .. (45:1.4) node[anchor=west]{$\bx_+$};
			\draw [blue,<-] (122.5:1.05) .. controls (115:1.6) and (125:1.3) .. (135:1.4) node[anchor=east]{$\bx_-$};

			\draw[very thin,gray] (90:0.03)  -- (90:1.2);

			\filldraw (0,1.2) circle (.6pt);
			\filldraw (0,0.6) circle (.6pt);
			\filldraw (0.56,0.82) circle (.6pt);
			\filldraw (-0.56,0.82) circle (.6pt);
			
			\draw[dashed] (-0.56,1.58) -- (1.68,0.08);
			\draw[dashed] (0.56,1.58) -- (-1.68,0.08);
			\draw[very thin,gray] (0.56,0.82)  -- (0,0);
			\node at (0,1.4) {$\bp$};

			\draw (0,0.6)  circle (0.6);

	\end{tikzpicture}
	\qquad
	\begin{tikzpicture}[scale=1.5]  
			
			\draw[fill=black!5] (0,0) circle (1);
			
			\draw [color=cyan!40!white,fill=cyan!40!white]
			(87:1.2) node[anchor=south,color=cyan!40!black] {$\qquad R_0$}
			-- (3:1.2)  arc (3:87:1.2)--cycle;
			\draw [color=white,fill=white](78.6:1.2)--(11.4:1.2) arc (11.4:78.6:1.2)--cycle;
			\draw (0.68,0.71) node[anchor=west]{$\scriptstyle\bx_0$};
			\filldraw (0.71,0.71) circle (.6pt);
			\draw (50:0.7) node[anchor=south,color=black] {$\ell_0\ $};
			\draw[<->,black] (45:0.03)  -- (45:0.87);
			\draw[black] (45:0.4) node[anchor=west]{$t$};
			\draw[<->] (270:0.03)  -- (270:0.97);
			\draw (270:0.4) node[anchor=west]{$1$};
			\draw [color=black]	(87:1.2)-- (3:1.2);
			\draw [blue,<-] (3:1.24) .. controls (6:1.5) and (9:1.8) .. (0:2) node[anchor=north]{$\bp_{0-}$};
			\draw [blue,<-] (87:1.24) .. controls (89:1.9) and (70:1.4) .. (60:1.5) node[anchor=west]{$\bp_{0+}$};
			\begin{scope}[rotate around={90:(0,0)}]
				\draw [color=cyan!40!white,fill=cyan!40!white]
				(87:1.2) -- (3:1.2)  arc (3:87:1.2)--cycle;
				\draw (80:1.2) node[anchor=east,color=cyan!40!black] {$R_1$};
				\draw [color=white,fill=white](78.6:1.2)--(11.4:1.2) arc (11.4:78.6:1.2)--cycle;
				\draw [color=green!80!black,fill=green!80!black]
				(72:1)--(19:1) arc (19:72:1)--cycle;
				\draw (46:0.72) node[color=green!60!black,anchor=north]{$U\cap R_1$};	
				\draw (0.69,0.71) node[anchor=south]{$\scriptstyle \bx_1$};
				\filldraw (0.71,0.71) circle (.6pt);
				\filldraw[blue] (3:1.2) circle (.3pt);
				\draw [blue,<-] (3:1.24) .. controls (4:1.5) and (9:1.5) .. (15:1.5) node[anchor=east]{$\bp_{1+}$};
				\filldraw[blue] (87:1.2) circle (.3pt);
				\draw [blue,<-] (87:1.24) .. controls (89:1.7) and (80:1.6) .. (70:1.6) node[anchor=south]{$\bp_{1-}$};
				
			\end{scope}
			
			\begin{scope}[rotate around={180:(0,0)}]
				\draw [color=cyan!40!white,fill=cyan!40!white]
				(87:1.2) node[anchor=north,color=cyan!40!black] {$R_2$\phantom{xxx}} 	-- (3:1.2)  arc (3:87:1.2)--cycle;
				
				\draw [color=white,fill=white](78.6:1.2)--(11.4:1.2) 	arc (11.4:78.6:1.2)--cycle;
				\draw (0.66,0.71) node[anchor=east]{{$\scriptstyle\bx_2$}};
				\filldraw (0.71,0.71) circle (.6pt);
				\draw [blue,<-] (87:1.23) .. controls (89:1.6) and 	(100:1.4) .. (100:1.4) node[anchor=west]{$\bp_{2-}$};
				\draw [blue,<-] (3:1.24) .. controls (4:1.3) and (5:1.4) .. (5:1.4) node[anchor=east]{$\bp_{2+}$};
				
			\end{scope}
			
			\begin{scope}[rotate around={270:(0,0)}]
				\draw [color=cyan!40!white,fill=cyan!40!white]
				(87:1.2) node[anchor=west,color=cyan!40!black] 	{$R_3$} 	-- (3:1.2)  arc (3:87:1.2)--cycle;
				
				\draw 	[color=white,fill=white](78.6:1.2)--(11.4:1.2) 	arc (11.4:78.6:1.2)--cycle;
				\draw (0.8,0.59) node[anchor=west]{$\scriptstyle\bx_3$};
				\filldraw (0.71,0.71) circle (.6pt);
				
			\end{scope}

			\filldraw (0,0) circle (.4pt);
			\draw (0,0)  circle (1);
			\draw (-0.5,-0.2) node[anchor=west] {$U$};
			\draw[color=red] (0,0) 	circle (1.2);
			\draw[color=red] (1.1,.5) node[anchor=west] {$U_\zeta=(1+r)U$};
			\filldraw[blue] (3:1.2) circle (.3pt);
			\filldraw[blue] (87:1.2) circle (.3pt);
			\begin{scope}[rotate around={90:(0,0)}]
				\filldraw[blue] (3:1.2) circle (.3pt);
				\filldraw[blue] (87:1.2) circle (.3pt);	
			\end{scope}
			\begin{scope}[rotate around={180:(0,0)}]
				\filldraw[blue] (3:1.2) circle (.3pt);
				\filldraw[blue] (87:1.2) circle (.3pt);	
			\end{scope}
			
		\end{tikzpicture}
	\end{center}
	\caption{	\label{fig:circle} Left: The construction in the proof of Proposition \ref{prop:Uzeta}. 
			The unit disc $U$ is shown in light gray.
	The disc $U_\zeta$ is bounded by the red circle. The two tangent lines (dashed) to the disc $U$ through the point $\bp$ are displayed. In addition, we depict the Thales circle containing the origin, the point $\bp$ and the  tangent points $\bx_\pm$.
	\\	Right: The sets $U$ and $U_\zeta$ are shown as on the left.  The strips $R_i$ are marked in light blue and have width $1-t$. The segment $\ell_0$ and $\bp_{i\pm}$, $i=0,1,2$ are indicated as well. Note that the parameters used in the proof are on an entirely different scale from those in the figure.}
\end{figure}

 To apply the methods from \cite{pach}, we consider the four points
\[
\bx_0=\tfrac{1}{\sqrt{2}}(1,1)^\top, \quad
\bx_1=\tfrac{1}{\sqrt{2}}(-1,1)^\top,  \quad½
\bx_2=\tfrac{1}{\sqrt{2}}(-1,-1)^\top,  \quad
\bx_3=\tfrac{1}{\sqrt{2}}(1,-1)^\top,
\]
on the boundary of $U$. 
Note that the unique outer normal of $U$ at $\bx_i$ is $\bx_i$, $i=0,\ldots,3$.
\medskip

For  $r\ge 0$ (associated to $U_{\zeta(r)}$)
inscribe a square $Q({r})$ in $U_{\zeta(r)}=(1+r)U$ such that its sides are parallel to the vectors $\bx_0,\bx_1,\bx_2,\bx_3$. 
We now fix the number  $0\le r<\sqrt 2-1$ to assure that $U\setminus Q(r)$ is not empty, and contains in particular the points $\bx_0,\bx_1,\bx_2,\bx_3$. For $t_r=(1+r)/\sqrt 2\le  1$ and a variable $t\in[t_r,1]$, to be determined later,  we now construct the sets $R_i=R_i(r,t)$,  by setting 
\[
R_i=\Big\{\bx\in (1+r)U: t\le \langle \bx,\bx_i\rangle \le1\Big\},
\] 
$i=0,\ldots,3$, see Fig.~\ref{fig:circle}.  If $t=1$, the set $R_i$ is a chord in $(1+r)U$ with midpoint $\bx_i$. For $t=t_r$, the sets $R_i$ are strips that meet any of their neighboring strips in one point each.

 The line segment  $\ell_i=\{\bx\in (1+r)U: \langle \bx,\bx_i\rangle =t\}$ is a subset of the boundary of $R_i$.   We will need the endpoints of $\ell_i=[\bp_{i+},\bp_{i-}]$, where  
 \begin{align}\label{def:pi}
 	\bp_{0\pm}=t\bx_0\pm w\bx_1,\qquad \bp_{1\pm}=t\bx_1\pm w\bx_0,
 	\qquad \bp_{2\pm}=-\bp_{0\mp},
 \end{align}
 $i=0,1,2$,
 where $w=w(r,t)=\sqrt{(1+r)^2-t^2}$. 
 
 Define $B_i=B\cap R_i$, $i=0,\ldots,3$, and recall that 
 $B_{\mathrm{out}}=B\setminus (1+r)U$.
The union of all lines  $g$ parallel to $\ell_i$ hitting $U\cap R_i$ is a strip of width $1-t$. Any line in this strip must hit $U$ and thus the barrier in $B_i\cup B_{\mathrm{out}}$, implying $|B_i\cup B_{\mathrm{out}}|\ge 1-t$. Since both sets are disjoint, we have $|B_i|\ge 2\eta$, $i=0,1,2,3$, with 
\begin{equation}\label{eq:etaout}
\eta=\tfrac{1}2\big(1-t-|B_{\mathrm{out}}|\big).
\end{equation}
The parameter $\eta$ is positive iff $|B_{\mathrm{out}}|<\frac{r}{1+r}(1-t)$, which is assumed from now on. 
If $\eta'$ is some positive number such that $\eta'\le \eta$, we trivially have $|B_i|\ge 2\eta'$ for $i=0,1,2,3$. Hence,  the following arguments all remain valid after replacing $\eta$ by $\eta'$. This will be exploited in Theorem \ref{thm:optDelta}, where a lower bound $\eta'$ of $\eta$ is obtained from an application of Proposition \ref{prop2}.

With literally the same arguments as in \cite[p.~18]{pach}, we obtain sets $R^+$ and $R^-$ (obtained from two of the $R_i$'s by applying the same rotation about the origin) and subsets $B^+$ and $B^-$ of barriers in them, such that   $|B^+|\ge \eta$ and $|B^-|\ge \eta$ and all segments in $B^+\cup B^-$ making an angle $\ge \lambda:=\pi/8$ with respect to the horizontal axis. 

If  $R^-$ and $R^+$ come from opposite $R_i's$, we may take $R_0$ and $R_2$, both rotated by $-\pi/4$, otherwise they come from 
neighboring $R_i's$ and we may assume that they are equal to $R_0$ and $R_1$, respectively. All sets $R_i$, and thus $R^+\cup R^-$ are contained in the disc $(1+r)U$ with diameter $D=2(1+r)$, but if they are neighboring, $D$ can be chosen smaller. It remains to discuss for which $t$ and $r$  they are separated by strips. \medskip

\emph{The case of opposing $R_i's$.}
We start the discussion with the case where the two sets $R_i$ found above are opposing each other, as this case turns out to be the critical one. We assume $R^+$ is equal to $R_0$, rotated with angle $-\pi/4$, and $R^-$ is equal to $R_2$, rotated with angle  $-\pi/4$ about $\bo$. 
Since $R^-\cup R^+\subset (1+r)U$ and $\|\bp_{0+}-\bp_{2-}\|=2(1+r)$ 
(these points are antipodal), $D=2(1+r)$ is the diameter of the union $R^-\cup R^+$. 
Let $\ell_0'$ and $\ell_2'$ be the rotation of   $\ell_0$ and $\ell_2$
 with angle $-\pi/4$ about $\bo$, respectively. More explicitly, \eqref{def:pi} shows that $\ell_i'=[\bp_{i+}',\bp_{i-}']$, $i=0,2$, with 
 \[
 \bp_{0\pm}'=(t,\pm w)^\top, \qquad \bp_{2\pm}'=(-t,\pm w)^\top, 
 \]
 where $w=\sqrt{(1+r)^2-t^2}$, as before.

 	\begin{figure}[h]
	\begin{center}
		\begin{tikzpicture}[scale=2]  
			
			\draw[fill=black!5] (0,0) circle (1);
			
			\draw [color=cyan!40!white,fill=cyan!40!white]
			(87:1.2) 
			-- (3:1.2) node[anchor=west,color=black] {$R_0$}  arc (3:87:1.2)--cycle;
			\draw [color=white,fill=white]
			(78.6:1.2)--(11.4:1.2)--(11.4:1.5) -- (78.6:1.5) --cycle;
			\draw (50:0.7) node[anchor=south,color=black] {$\ell_0$};
			\draw [color=black]	(87:1.2)-- (3:1.2);
			\begin{scope}[rotate around={90:(0,0)}]
				\draw [color=cyan!40!white,fill=cyan!40!white]
				(87:1.2) -- (3:1.2)  arc (3:87:1.2)--cycle;
				\draw (80:1.2) node[anchor=north east] {$R_1$};
				\draw [color=white,fill=white]
				(78.6:1.2)--(11.4:1.2)--(11.4:1.5) -- (78.6:1.5) --cycle;
				\draw (50:0.7) node[anchor=south,color=black] {$\ell_1$};
				\draw [color=black]	(87:1.2)-- (3:1.2);
			\end{scope}

			\filldraw (0,0) circle (.4pt);
			\draw (0,0)  circle (1);
			\draw (0.2,-0) node[anchor=west] {$U$};
			
			\draw [color=black,dashed]	(0.07,1.4)-- (-.7,-1.2) node[anchor=west] {$g$};
			
		\end{tikzpicture}
	\end{center}
	\caption{	\label{fig:separate}
		A line $g$ separates the strips $R_0$ and $R_1$ if and only if it separates the line segments $\ell_0$ and $\ell_1$.}
\end{figure}

A line $g$  separates $R^-$ and $R^+$ if and only if it separates $\ell_0'$ and $\ell_2'$,  see Fig.~\ref{fig:separate}.
The slope of a separating line must thus be between the slopes of the line $g_+$ through $\bp_{0+}'$ and $\bp_{2-}'$,  and the line $g_-$ through $\bp_{0-}'$ and $\bp_{2+}'$. Since
\begin{equation}
	\label{eq:diff_diagon}
	\bp_{0+}'-\bp_{2-}'=2(t,w)^\top,
\end{equation}  the unit vector $\tfrac{1}{1+r}(t,w)^\top$ is parallel to $g_+$, which thus  makes angle 
\[
\alpha_+=\arccos\Big(\frac{t}{1+r}\Big)
\]
with the horizontal axis. 
By symmetry with respect to the vertical axis, the corresponding angle for $g_-$ is $\alpha_-=\pi-\alpha_+$. Hence,  $R^-$ and $R^+$ are $\lambda$-separated with $\lambda=\pi/8$ 
iff $\alpha_+<\pi/8$, i.e.~iff
\begin{equation}
	\label{eq:lambda_necessary.new}
	t>(1+r)\cos\tfrac{\pi}{8}=\tfrac12 \sqrt{2+\sqrt2}(1+r)\approx 0.924(1+r). 
\end{equation}
This is stronger than the original condition $t\ge t_r=\tfrac12\sqrt{2}(1+r)\approx0.707(1+r)$. Since $t\le 1$, we also must have 
\begin{equation}
	\label{eq:r_necessary.new}
	r<  \frac1{\cos\frac\pi8}-1=  \frac{2}{\sqrt{2+\sqrt2}}-1\approx 0.082.
\end{equation}
In the following, we assume that  \eqref{eq:lambda_necessary.new} and \eqref{eq:r_necessary.new} hold. 
We now aim for the largest $0<\gamma<\lambda$ such that strips of width $W=\eta\sin(\gamma)$ and angles $\pm(\lambda-\gamma)$ separate $\ell_0'$ and $\ell_2'$, i.e.~such that 
\begin{equation}\label{eq:opt_gamma.new}
	\bp_{0+}'(\lambda-\gamma)-\bp_{2-}'(\lambda-\gamma)\ge \eta\sin \gamma. 
\end{equation}
In view of \eqref{eq:diff_diagon}, the trigonometric addition theorem and the fact that we have chosen $\lambda=\pi/8$, the left hand side of \eqref{eq:opt_gamma.new} is 
\begin{align*}
2[t\sin(\lambda-\gamma)-w\cos(\lambda-\gamma)]&=
2[(t\sin\tfrac{\pi}8-w\cos\tfrac{\pi}8)\cos\gamma-
(t\cos\tfrac{\pi}8+w\sin\tfrac{\pi}8)\sin\gamma]. 
\end{align*}

Thus, \eqref{eq:opt_gamma.new} is equivalent to 
\[
\eta\le2[
(t\sin\tfrac{\pi}8-w\cos\tfrac{\pi}8)\cot\gamma-
(t\cos\tfrac{\pi}8+w\sin\tfrac{\pi}8)],
\]
or 
\begin{equation}\label{eq:best_gamma.new}
	\cot\gamma\ge  \frac{\eta+2t\cos\tfrac{\pi}8+2w\sin\tfrac{\pi}8}{2t\sin\tfrac{\pi}8-2w\cos\tfrac{\pi}8}=:h,
\end{equation}
where we used that  $2t\sin\tfrac{\pi}8-2w\cos\tfrac{\pi}8>0$ holds, as this inequality is equivalent to   \eqref{eq:lambda_necessary.new}.  
Since the cotangent function is strictly decreasing on $[0,\pi]$, the largest $\gamma$ satisfying \eqref{eq:best_gamma.new} satisfies this  inequality with equality and thus is equal to $\gamma^*=\mathrm{arccot} (h)$. Again, due to symmetry with respect to the vertical axis, the same is true in direction  with angle $-(\lambda- \gamma)$.

Since $\sin\gamma=(1+\cot^2\gamma)^{-1/2}$, the choice $\gamma=\gamma^*$ yields 
\begin{equation}
	\label{eq:W=W(eta,r,t).new}
	W=\frac{\eta}{\sqrt{1+h^2}}, 
\end{equation}
and Corollary \ref{coro2} (with $D=2(1+r)$, which we saw is best possible) implies 
\begin{equation}\label{eq:forDeltaWithBout}
	|B|-\pi\ge \frac{W^2}{2D}= \frac{\eta^2}{1+h^2}\frac{1}{4(1+r)},
\end{equation}
where the right-hand side depends on the parameters 
$r\in(0,\infty)$, $t\in (0,1]$ and $|B_{\mathrm{out}}|$, obeying \eqref{eq:lambda_necessary.new}, \eqref{eq:r_necessary.new}, and 
$|B_{\mathrm{out}}|<\frac{r}{1+r}(1-t)$.

\begin{thm}\label{thm:optDelta}
	Let $B$ be a straight barrier for the unit disc $U$ and $\delta=|B|-\pi$. 
		For all   $0\le r< \frac1{\cos\frac\pi8}-1$, 
	$\cos\frac\pi8 (1+r)< t\le 1$ the following statement holds. 
	If  
	\begin{equation}\label{eq:bloed} 
	\eta'= \tfrac{1}2\big(1-t-\tfrac{1+r}r\delta\big)\ge0, 
	\end{equation}
	then 
	\begin{equation}\label{eq:forDeltaWithBout'new}
		4(1+r)(1+(h')^2)\delta\ge (\eta')^2,
	\end{equation}	
 where 
	\[
	h'=\frac{\eta'+2t\cos\tfrac{\pi}8+2w\sin\tfrac{\pi}8}{2t\sin\tfrac{\pi}8-2w\cos\tfrac{\pi}8},
	\]
	(this is definition \eqref{eq:best_gamma.new} with $\eta'$ replacing $\eta$) and $w=\sqrt{(1+r)^2-t^2}$. 
\end{thm}
\begin{proof}
	The arguments  follow along the lines of the previous proof. 
	Proposition \ref{prop2} implies that $\eta'\ge 0$ is a lower bound for $\eta$. Since the claim is trivial for $\eta'= 0$, 
	we may assume  $\eta'>0$, so all the arguments can be repeated using $\eta'$. 
	One obtains 	
	\begin{equation*}
		\delta\ge \frac{(\eta')^2}{1+(h')^2}\frac{1}{4(1+r)},
	\end{equation*}	
	which shows the assertion after rearranging. 	
\end{proof}

Both sides of \eqref{eq:forDeltaWithBout'new} depend on the unknown $\delta$. But  \eqref{eq:forDeltaWithBout'new} is not satisfied when $\delta=0$, so by continuity there is a smallest $\delta=\delta(r,t)>0$ such that equality holds in \eqref{eq:forDeltaWithBout'new}. 
The largest $\delta^*=\max_{r,t}\delta(r,t)$ can be searched numerically within the ranges of $r$ and $t$, since the inequality amounts in an inequality for a cubic expression in $\delta$. 
We did this\footnote{We can provide the code upon request.} and obtained the parameters 
\begin{equation*}
	r_0:=0.001067, \qquad t_0:= 0.965763, 
\end{equation*}
and thus
$$\delta^* \geq \delta(r_0,t_0) = 1.076457 \cdot 10^{-6}. $$

The corresponding variable in \eqref{eq:bloed} is 
\begin{equation}\label{eq:eta0p}
\tfrac{1}2\big(1-t_0-\tfrac{1+r_0}{r_0}\delta(r_0,t_0)\big)\ge
0.016613=:\eta'_0\ge 0.
\end{equation}

That  $\delta(r_0,t_0)$ indeed is a lower bound for $|B|-\pi$
despite the fact that \eqref{eq:forDeltaWithBout'new} only holds conditioned on \eqref{eq:bloed}, can be seen by contradiction: 
for the fixed values $(r,t)=(r_0,t_0)$
suppose that $|B|-\pi<\delta(r_0,t_0)$. Then, the corresponding $\eta'$ satisfies 
\[
	\eta'= \tfrac{1}2\big(1-t_0-\tfrac{1+r_0}{r_0}(|B|-\pi)\big)\ge
	\tfrac{1}2\big(1-t_0-\tfrac{1+r_0}{r_0}\delta(r_0,t_0)\big)\ge \eta'_0\ge 0, 
\]
so \eqref{eq:forDeltaWithBout'new} must be satisfied with $\delta=|B|-\pi$. Since  \eqref{eq:forDeltaWithBout'new}  is violated for $\delta=0$,  the intermediate value theorem  guarantees the existence of $\delta'$ with $0<\delta'\le |B|-\pi<\delta(r_0,t_0)$ such that \eqref{eq:forDeltaWithBout'new} holds with equality for $\delta=\delta'$, contradicting the definition of $\delta(r_0,t_0)$. 
\smallskip

From now on we fix the parameters $\lambda=\pi/8$, $r=r_0$ and $t=t_0$ and exploit the fact that $|B_i|\ge \eta_0'$, $i=0,1,2,3$, where $\eta_0'$ is given in \eqref{eq:eta0p}.  
It remains to show that also in the case of  neighboring sets $R_i$, 
this parameter choice leads to a deviation from Jones' bound at least of size $ 1.076457 \cdot 10^{-6}$. 
\medskip

\emph{The case of neighboring $R_i's$.} 
Suppose  that  $R^+$ and $R^-$ are  $R_0$ and $R_1$, respectively. 
We choose\footnote{The choice  $\gamma=0.124871$ would be optimal to maximize the width $W$ given $(r,t)$. This  can be shown with arguments similar to those in the first case. We do not need the optimality property in the line of arguments and thus work with a simpler constant.} $\gamma=0.124$ and claim that $R_+$ and $R_-$ 
can be separated by strips with angle $\pm(\lambda - \gamma)$ and width $W=\eta \sin \gamma$. A line $g$  separates $R_0$ and $R_1$ if and only if it separates $\ell_0$ and $\ell_1$. This and symmetry considerations show that the claim is equivalent to 
\begin{equation}\label{eq:opt_gamma1}
	\bp_{0+}(\lambda-\gamma)-\bp_{1-}(\lambda-\gamma)\ge \eta_0'\sin \gamma. 
\end{equation}
The relations in \eqref{def:pi} imply $\bp_{0+}-\bp_{1-}=\sqrt2 (t,w)^\top$, so 
\begin{align*}
\bp_{0+}(\lambda-\gamma)-\bp_{1-}(\lambda-\gamma)&= 
\sqrt2(t\sin(\lambda-\gamma)-w\cos(\lambda-\gamma))
\\&> 0.002169\ge\eta_0'\sin \gamma, 
\end{align*}
confirming \eqref{eq:opt_gamma1}. Corollary \ref{coro2} now implies 
\begin{equation}\label{eq:generalD}
| B | -p\geq  \frac{(\eta_0' \sin \gamma)^2}{2D} > \frac{2.111\cdot 10^{-6}}{D}. 
\end{equation}
As in the case of the unit square, a tight choice of $D$ is required now. The trivial bound $D\le 2(1+r)$ would yield $| B | -p\geq 1.054\cdot 10^{-6}$, which is not sufficient. As in the case of the unit disc, $D$   can be chosen smaller: Let $K$ be a disc centered at the midpoint 
$
\bm=(1/\sqrt2)(0,t-w)^\top
$ 
of $\bp_{0-}$ and $\bp_{1-}$ with radius $\rho:=\|\bm-\bp_{0-}\|=(t+w)/\sqrt2$. It is easy to check that $K$ contains $R_0\cup R_1$, and thus $D=2\rho$ can be chosen in \eqref{eq:generalD}, yielding finally 
\begin{equation*}
	| B | -p\geq  \frac{2.111\cdot 10^{-6}}{ (\sqrt{2}(t+w))} \ge 1.214\cdot 10^{-6}.   
\end{equation*}
Thus, also in this case, the deviation from Jones' bound must be larger than the claimed lower bound and the proof is complete. 
\medskip 

We conclude with two remarks. One could (and we did) also optimize the value of $(r,t)$ in the case of neighboring $R_i's$. Due to the better choice of $D$, however, the case with opposing  $R_i's$ turns out to be the 
critical one. This  is the reason why we focused on optimizing the latter. 

A second remark: As in the case of the unit square, one might suspect that the main result in \cite{steinerberger} can be used to improve the lower bound for the barrier lengths of the unit disc. However, the main result \cite[Theorem]{steinerberger} for $U$ implies that a barrier $B$ with a length close to Jones' bound has an angular orientation measure $\mu_B$ that is approximately uniform on $[0,2\pi)$ and there is no simple way to use this information for barrier parts close to the points $\bx_0,\ldots,\bx_3$.

 

\end{document}